\documentclass[12pt]{article}
\usepackage{fullpage,graphicx,psfrag,amsmath,amssymb,amsthm,amsfonts,verbatim}
\usepackage[small,bf]{caption}
\usepackage{cite,bm}
\usepackage{array}
\newcolumntype{M}[1]{>{\centering\arraybackslash}m{#1}}
\usepackage{url}
\usepackage{multirow}
\usepackage[redeflists]{IEEEtrantools}
\newtheorem{theorem}{Theorem}
\newtheorem{assumption}{Assumption}
\newtheorem{lemma}{Lemma}

\title{A Convex Primal Formulation for Convex Hull Pricing}
\author{Bowen Hua \and Ross Baldick\footnote{The authors are with the Department of Electrical and Computer Engineering, the University of Texas at Austin, Austin, TX 78701 USA (e-mail: bhua@utexas.edu; baldick@ece.utexas.edu). This work was supported by the Defense Threat Reduction Agency (DTRA) under Grant HDTRA1-10-1-0087 and Grant HDTRA1-14-1-0021.
The initial version of this paper was submitted to arXiv on May 17, 2016.}}

\begin{document}
\maketitle

\begin{abstract}
In certain electricity markets, because of non-convexities that arise from their operating characteristics, generators that follow the independent system operator's (ISO's) decisions may fail to recover their cost through sales of energy at locational marginal prices. The ISO makes discriminatory side payments to incentivize the compliance of generators. Convex hull pricing is a uniform pricing scheme that minimizes these side payments. The Lagrangian dual problem of the unit commitment problem has been solved in the dual space to determine convex hull prices. However, this approach is computationally expensive. We propose a polynomially-solvable primal formulation for the Lagrangian dual problem. This formulation explicitly describes for each generating unit the convex hull of its feasible set and the convex envelope of its cost function. We cast our formulation as a second-order cone program when the cost functions are quadratic, and a linear program when the cost functions are piecewise linear. A 96-period 76-unit transmission-constrained example is solved in less than fifteen seconds on a personal computer.
\end{abstract}

\section{Introduction}
%
%
%
%
Day-ahead and some real-time electricity markets in the US currently base their market clearing model on a unit commitment and economic dispatch (UCED) problem \cite{PJM2015,ERCOT2016,Sioshansi2010}. The independent system operator (ISO) sends energy prices and target quantity instructions to each generating unit (called ``unit'' for short hereafter) based on a welfare-maximizing solution to the UCED problem. Ideally, energy prices provide incentives for profit-maximizing market participants to comply with the ISO's commitment and dispatch decisions. Various issues prevent this ideal, however, including the non-convexities in the market that arise from units' operating characteristics. Consequently, start-up and no-load costs of units may not be covered by sales of energy at locational marginal prices (LMPs). More generally, in a market with non-convexities, there might be no set of uniform prices\footnote{A uniform price is a single price that applies to all transactions at a given bus. Prices may vary locationally.} that supports a welfare-maximizing solution \cite{Gribik2007}.\footnote{A set of prices is said to support a solution if the economic agents' profit maximizing decisions align with this solution.}

One way to address this problem is to maintain uniform energy prices based on marginal energy costs and provide side payments to units that have an incentive to deviate from the ISO's solution. This side payment is also known as an ``uplift'' payment. In principle, the amount of uplift payment to a unit should cover its lost opportunity cost, the gap between its maximum possible profit and the actual profit obtained by following the ISO's solution. For example, the cost of a fast-start unit dispatched at its minimum limit may not be covered by sales of energy at its marginal cost. If revenues are based solely on LMP, then the profit-maximizing decision of this unit is to shut down. An uplift payment is needed to keep this unit online \cite{Wang2016}. Unlike energy prices, uplift payments are non-uniform (discriminatory) in that the amount of payment is unit-specific. These side payments make it harder for a potential entrant to determine if new entry would be profitable, particularly if the uplift payments are not disclosed publicly. 

Transparency of the market can be improved by keeping uplift payments as low as possible. To this end, several alternative pricing schemes have been proposed. For example, an ad-hoc method to reduce uplift payments to fast-start units is to relax their minimum generation limits to zero, so that they can set the LMPs \cite{FERC2015}. Keeping marginal prices as uniform energy prices, the pricing scheme proposed in \cite{o2005efficient} introduces artificial constraints that set the commitment variables at a welfare-maximizing solution and create discriminatory side payments for commitment decisions based on the optimal dual variables associated with these artificial constraints. Different from the pricing schemes that aim at supporting a welfare-maximizing solution, pricing schemes such as \cite{ruiz2012pricing} have been proposed to incentivize a commitment and dispatch solution that is close, but not necessarily equal to the ISO's welfare-maximizing solution. In these methods, allocative efficiency is traded off against transparency. Instead of focusing on unit-commitment based markets that are typical in the US, Ortner and Huppmann \cite{Ortner} define a quasi-equilibrium for a self-committed electricity market, and determine prices through a mathematical program with equilibrium constraints.  See \cite{Liberopoulos2016} for a comprehensive review of different pricing schemes for markets with non-convexities.

Convex hull pricing \cite{Ring1995,Hogan2003,Gribik2007} is a pricing scheme that minimizes certain uplift payments over all possible uniform prices, and has received much attention. The Midcontinent ISO (MISO) has implemented an approximation of convex hull pricing, and refers to convex hull prices as extended locational marginal prices (ELMPs) \cite{Wang2016}. Convex hull prices are slopes of the convex envelope of the system cost function,\footnote{More precisely, they are sub-gradients of the convex envelope of the system cost function. The system cost function here is the value function of the UCED problem parametrized by demand. The convex envelope of a function is the largest convex under-estimator of the given function.} and are thus non-decreasing with respect to demand. These prices minimize the total uplift payment defined by the duality gap between the UCED problem and its Lagrangian dual.

The Lagrangian dual problem of the UCED problem in which the system-wide constraints are dualized has been used to determine convex hull prices. This problem is convex but non-smooth. Algorithms such as sub-gradient methods, bundle methods \cite{lemarechal2001lagrangian}, and cutting plane methods \cite{goffin2002convex} have been proposed to solve the Lagrangian dual problem of mixed-integer programming problems. In the context of convex hull pricing, the focus is on the optimality of dual variables, rather than obtaining primal solutions. Therefore, in addition to general-purpose methods, an outer approximation method \cite{Gribik2007}, a sub-gradient simplex cutting plane method \cite{Wang2013a}, and an extreme-point sub-differential method \cite{Wang2013} have been designed specifically for convex hull pricing.
 
None of the above-mentioned methods guarantees convergence in polynomial time. For non-smooth optimization techniques like the sub-gradient method for which no certificate of optimality exists, the algorithm is often terminated before an optimal value is attained \cite[Chapter 10.3]{wolsey1998integer}. Obtaining exact dual maximizers of the Lagrangian dual problem is computationally expensive \cite{Hogan2014a, Wang2016}. Consequently, MISO implements a single-period approximation of convex hull pricing that is based on a version of the UCED problem in which integer variables are relaxed to being continuous \cite{Wang2016}.

Efficient computation of convex hull prices remains challenging. This paper proposes a polynomially-solvable formulation for convex hull pricing. Section \ref{sec:chp} introduces convex hull pricing. Section \ref{sec:formulation} presents a primal formulation of the Lagrangian dual problem of the UCED problem. In this formulation, we explicitly describe convex hulls of individual unit's feasible commitment and dispatch decisions and convex envelopes of individual cost functions. We cast the primal formulation as a second-order cone program if the cost functions are quadratic, and a linear program if the cost functions are affine or piecewise linear. Our formulation gives exact convex hull prices in the absence of ramping constraints. 

Section \ref{sec:extensions} considers several extensions to our model. We show that exact convex hull prices can still be obtained when ancillary services or any linear system-wide constraints (such as the transmission constraints) are introduced. Ramping constraints lead to an exponential number of valid inequalities in the convex hull representation. In this case, we approximate the convex hulls using valid inequalities developed in \cite{Damc-Kurt2015, Pan2015}. Since only a linearly-constrained convex program needs to be solved to obtain convex hull prices, our approach is robust and scalable. Section \ref{sec:results} reports numerical tests conducted on several examples found in the literature, and conclusions follow in Section \ref{sec:conclusions}.

\section{Convex Hull Pricing}

\label{sec:chp}
This section formulates a unit commitment and economic dispatch (UCED) problem. We refer to coupling constraints that enforce system-wide requirements (e.g., demand and transmission constraints) as system-wide constraints. In contrast, constraints that an individual unit faces are called private constraints. 

We first consider a UCED problem in which the only type of system-wide constraint is supply-demand balance constraint. We define the uplift payment to a unit to be its lost opportunity cost. We then introduce the concept of convex hull pricing and describe the Lagrangian dual problem of the UCED problem. This simple model suffices to illustrate the basic ideas of convex hull pricing and our primal formulation.

\subsection{The UCED Problem}
 We consider a $T$-period offer-based UCED problem with $|\mathcal{G}|$ units. For unit $g \in \mathcal{G}$ at time $t \in \lbrace 1, \ldots, T \rbrace$, the commitment variable $x_{gt}$ is $1$ if the unit is online and is $0$ if the unit is offline. The start-up variable $u_{gt}$ is $1$ if unit $g$ starts up at period $t$ and is $0$ otherwise. 
 
Denote unit $g$'s dispatch-level (dispatched power output) vector by $\bm{p}_g \in \mathbb{R}^{T}_+$, whose $t$-th component is the dispatch level at time $t$. Similarly, $\bm{x}_g \in \lbrace 0,1 \rbrace^{T}$ denotes the commitment vector, and $\bm{u}_g \in \lbrace 0,1 \rbrace^{T-1}$ is the start-up vector.\footnote{The start-up variables $\bm{u}_g$ are defined from period $2$ to $T$. For simplicity we do not consider the initial conditions of the units. Our formulation can be extended to consider these conditions.} Let unit $g$'s offer cost function be $C_g(\bm{p}_g, \bm{x}_g, \bm{u}_g)$, which may include energy, start-up, and no-load costs. As in \cite{Guan1992,Ostrowski2012}, we assume that $C_g$ is convex piecewise linear or convex quadratic in $\bm{p}_g$. We assume that the start-up and no-load costs are constant. Let $\mathcal{X}_g \subseteq \mathbb{R}_+^{T} \times \lbrace 0,1 \rbrace^{T} \times \lbrace 0,1 \rbrace^{T-1}$ be the set of feasible commitment and dispatch decisions for unit $g$. We assume that private constraints that define $\mathcal{X}_g$ are specified by linear inequalities; these constraints may include generation limits, minimum up/down time, and perhaps ramping constraints \cite{Guan1992}. 
 
Let $\bm{d} \in \mathbb{R}^{T}_+$ be a demand vector, whose $t$-th component denotes system demand at time $t$. The UCED problem makes a set of commitment and dispatch decisions that minimizes the total cost, while satisfying physical and operational constraints:
\begin{IEEEeqnarray}{llllr}
v(\bm{d}) = \quad &\underset{\bm{p}_g, \bm{x}_g, \bm{u}_g, \, g \in \mathcal{G}}{\min} & \quad & \sum_{g \in \mathcal{G}} C_g(\bm{p}_g, \bm{x}_g, \bm{u}_g)& \label{opt:UCED_start}\\
&\text{s.t.} && \sum_{g \in \mathcal{G}} \bm{p}_g = \bm{d} \label{opt:UCED_demand} \\
&&&(\bm{p}_g, \bm{x}_g, \bm{u}_g) \in \mathcal{X}_g & \quad \forall g \in \mathcal{G} \label{opt:UCED_end}.
\end{IEEEeqnarray}

In addition, we view the UCED problem as parametrized by the demand vector $\bm{d}$, and denote the value function of the UCED problem by $v(\bm{d})$.

\subsection{Non-convexity and Uplift Payments}

Suppose that an energy price vector $\bm{\pi}$ is specified by the independent system operator (ISO). Assume that unit $g$ is a price-taker.\footnote{We assume that the units are prices takers in the economic sense that they cannot affect prices.} Its profit maximization problem is:
\begin{IEEEeqnarray}{llllr}
w_g(\bm{\pi})  = \quad &\underset{\bm{p}_g, \bm{x}_g, \bm{u}_g}{\max} & \quad & \bm{\pi}^\intercal \bm{p}_g -  C_g(\bm{p}_g, \bm{x}_g, \bm{u}_g)& \label{opt:profitmax_start} \\
& \text{s.t.} && (\bm{p}_g, \bm{x}_g, \bm{u}_g) \in \mathcal{X}_g \label{opt:profitmax_end},
\end{IEEEeqnarray}
where $w_g(\bm{\pi})$ is the value function of this problem.

From a microeconomic viewpoint, the ISO's UCED problem \eqref{opt:UCED_start}--\eqref{opt:UCED_end} is a social planner's problem whose solution is welfare-maximizing. Problem \eqref{opt:profitmax_start}--\eqref{opt:profitmax_end} is the profit maximization problem of a rational agent. If the ISO's problem is convex and satisfies strong duality, we can set $\bm{\pi}$ to be the optimal dual vector associated with the supply-demand balance constraints in the planner's problem. As a result, there exist individual profit-maximizing decisions that align with the welfare-maximizing solution, but the ISO may need to specify which solution is welfare maximizing. If the ISO's problem were strictly convex, then prices alone would provide sufficient information for the units to determine an efficient decision \cite[Chapter 16]{mas1995microeconomic}. 

However, the UCED problem is non-convex because of the integer decision variables $\bm{x}_g$ and $\bm{u}_g$. Thus, in general, there does not exist a set of prices that support the ISO's decisions. In particular, revenues from locational marginal prices (LMPs) may not cover the offered costs of a unit, and the unit may prefer to deviate from the ISO's commitment and dispatch decisions, unless there is additional incentive to follow those decisions. LMPs are determined as the optimal dual variables associated with the supply-demand balance constraints in a continuous convex economic dispatch problem with commitment decisions fixed at ISO-determined optimal values.

One way to address the above-mentioned problem is for the ISO to maintain uniform energy prices and provide side payments to units whose individually rational decision is different from ISO's. In principle, these side payments, also known as ``uplift,'' should cover the gap between the maximum possible profit (the optimal objective function value of problem  \eqref{opt:profitmax_start}--\eqref{opt:profitmax_end}), and the profit made by following the ISO's decision (the value of objective function \eqref{opt:profitmax_start} evaluated at the ISO's decisions).

Mathematically, given a set of uniform energy prices $\bm{\pi}$ and an ISO's welfare-maximizing decision $(\bm{p}^*, \bm{x}^*, \bm{u}^*)$, the amount of uplift payment needed for unit $g$ equals its \emph{lost opportunity cost}: 
\begin{equation}
	U_g(\bm{\pi}, \bm{p}^*, \bm{x}^*, \bm{u}^*) = w_g(\bm{\pi}) - (\bm{\pi}^\intercal \bm{p}_g^* -  C_g(\bm{p}_g^*, \bm{x}_g^*, \bm{u}_g^*)) , \label{eq:individual_uplift}
\end{equation}
Since the ISO's decision $(\bm{p}^*, \bm{x}^*, \bm{u}^*)$ is a feasible, but not necessarily optimal solution to problem  \eqref{opt:profitmax_start}--\eqref{opt:profitmax_end}, $U_g(\bm{\pi}, \bm{p}^*, \bm{x}^*, \bm{u}^*)$ is non-negative. 


\subsection{Convex Hull Pricing and the Lagrangian Dual Problem}

Convex hull prices can be determined through a Lagrangian relaxation \cite{lemarechal2001lagrangian} of the UCED problem. We dualize the supply-demand balance constraints \eqref{opt:UCED_demand} and obtain the Lagrangian dual function 
\begin{equation}
	q(\bm{\pi}) =  \sum_{g \in \mathcal{G}} \left( \underset{(\bm{p}_g, \bm{x}_g, \bm{u}_g) \in \mathcal{X}_g}{\min} C_g(\bm{p}_g, \bm{x}_g, \bm{u}_g) - \bm{\pi}^\intercal \bm{p}_g \right) + \bm{\pi}^\intercal \bm{d}, \label{eq:Lagrangian_dual_function}
\end{equation}
where $\bm{\pi} \in \mathbb{R}^{T}$ is now the dual vector associated with the supply-demand balance constraints. The Lagrangian dual problem is
\begin{equation}
	\underset{\bm{\pi}}{\max} \quad q(\bm{\pi}). \label{opt:L_dual}
\end{equation}

The \emph{convex hull prices} are defined to be the dual maximizers $\bm{\pi}^*$. The value function of the Lagrangian dual problem \eqref{opt:L_dual} as a function of $\bm{d}$ is the convex envelope of $v(\bm{d})$ \cite{Falk1969}. The price vector is a sub-gradient of the convex envelope of $v(\bm{d})$. 

The duality gap between the UCED problem and its dual \eqref{eq:Lagrangian_dual_function} is exactly the total lost opportunity costs,
\begin{equation}
	 \sum_{g \in \mathcal{G}} U_g(\bm{\pi}, \bm{p}^*, \bm{x}^*, \bm{u}^*). \label{eq:total_uplift}
\end{equation}
Consequently, convex hull pricing as a uniform pricing scheme minimizes total uplift payment as defined by \eqref{eq:total_uplift}; that is, it minimizes the total lost opportunity costs of all participating units. This is a special case of a more general result on the type of uplift payments convex hull pricing minimizes, as shown in \cite{Gribik2007, Cadwalader2010, Schiro2015}, since the only type of system-wide constraint considered for now is the supply-demand balance constraint. 

\section{A Primal Formulation for Convex Hull Pricing}
\label{sec:formulation}

This section proposes a convex primal formulation for the Lagrangian dual problem of the UCED problem. In this formulation, the feasible set for each unit is replaced by its convex hull, and the individual cost functions are replaced by their convex envelopes.

\subsection{A Primal Formulation for the Lagrangian Dual Problem}
We make the following:
\begin{assumption}
The set $\mathcal{X}_g$ is compact for all $g \in \mathcal{G}$, and all system-wide constraints are linear.\label{assump}
\end{assumption}

Let $\text{conv}( \cdot )$ denote the convex hull of a set.\footnote{The convex hull of a set is all convex combinations of points in that set.} Let $C_{g, \mathcal{X}_g}^{**}(\cdot)$ be the convex envelope of $C_g(\cdot)$ \emph{taken over $\mathcal{X}_g$}. The function $C_{g, \mathcal{X}_g}^{**}(\cdot)$ is the largest convex function on $\text{conv}(\mathcal{X}_g)$ that is an under-estimator of $C_g$ on $\mathcal{X}_g$. It is also the conjugate of the conjugate of $C_g$. 

Note that the UCED problem is separable across $g$ absent the system-wide constraints. We have:
\begin{theorem}
	Under Assumption \ref{assump}, (a) the optimal objective function value of the Lagrangian dual problem \eqref{opt:L_dual} equals the minimum of the following problem denoted by \textbf{\emph{CHP-Primal}}:
	\begin{IEEEeqnarray}{lllr}
\underset{\bm{p}_g, \bm{x}_g, \bm{u}_g, g \in \mathcal{G}}{\min} & \quad & \sum_{g \in \mathcal{G}} C_{g, \mathcal{X}_g}^{**}(\bm{p}_g, \bm{x}_g, \bm{u}_g)& \label{opt:CHP_start}\\
\rm{s.t.} && \sum_{g \in \mathcal{G}} \bm{p}_g = \bm{d} \label{opt:CHP_demand} \\
&&(\bm{p}_g, \bm{x}_g, \bm{u}_g) \in \emph{\text{conv}}(\mathcal{X}_g)  & \quad \forall g \in \mathcal{G} \label{opt:CHP_end},
\end{IEEEeqnarray}
and (b) an optimal dual vector associated with \eqref{opt:CHP_demand} is an optimal solution to \eqref{opt:L_dual}.
\label{thm:primal_formulation}

\end{theorem}
\begin{IEEEproof}
Since the private and system-wide constraints are defined by linear equalities and inequalities, strong duality holds between \textbf{CHP-Primal} and its Lagrangian dual problem. Therefore, Theorem \ref{thm:primal_formulation} holds by Theorem 3.3 in \cite{Falk1969}.
\end{IEEEproof}

Theorem \ref{thm:primal_formulation} suggests that if we have an explicit characterization of $C_{g, \mathcal{X}_g}^{**}(\bm{p}_g, \bm{x}_g, \bm{u}_g)$ and $\textrm{conv}(\mathcal{X}_g)$, we can solve \textbf{CHP-Primal} to obtain the dual maximizers of \eqref{opt:L_dual}. The convex hull prices are the optimal dual variables associated with the supply-demand balance constraints \eqref{opt:CHP_demand}. Schiro et al. \cite{Schiro2015} present an equivalent primal formulation. Van Vyve \cite{VanVyve2011} describes another primal formulation that is restricted to affine cost functions.

\subsection{Characterization of the Convex Hulls}

The convex hull of a compact set defined by linear inequalities and integrality requirements is a bounded polyhedron. In general, it is difficult to obtain an explicit description of the convex hull of a mixed-integer set defined by arbitrary linear constraints. The number of valid inequalities needed is typically exponential in the size of the input \cite{Cornuejols2007}. A general-purpose method proposed in \cite{Balas1998} is used in \cite{Schiro2015} to obtain a convex hull description of a unit's feasible set. This method applies to a feasible set defined by arbitrary linear constraints.  All feasible commitment decisions are enumerated in this method, and both the number of variables and number of constraints in the resulting description are exponential in the number of time periods. In addition, Van Vyve \cite{VanVyve2011} presents a more compact convex hull description of a unit's feasible set using $\mathcal{O}(T^3)$ variables.

Recent polyhedral studies of a unit's feasible set \cite{rajan2005minimum, Damc-Kurt2015, Pan2015} exploit its structure. We use special-purpose valid inequalities proposed in these studies to obtain a tractable description of $\text{conv}(\mathcal{X}_g)$, using only $\mathcal{O}(T)$ variables. 

Let $L_g$ and $l_g$, respectively, be the minimum up and minimum down times for unit $g$, and let $\underline{p}_{g}$ and $\overline{p}_{g}$ be the minimum and maximum generation levels for unit $g$. We consider feasible commitment and dispatch decisions of a unit limited by:
\begin{itemize}
\item state-transition constraints that represent the relationship between binary variables:
\begin{equation}
	u_{gt} \geq x_{gt} - x_{g,t-1}, \quad \forall t \in [2,T],\label{eq:state_transition}
\end{equation}
\item minimum up/down time constraints \cite{rajan2005minimum}:
\begin{IEEEeqnarray}{lll}
	\sum_{i = t - L_g + 1}^t u_{gi} \leq x_{gt}, &\forall t \in [L_g + 1, T], \label{eq:IBM_1}\\
\sum_{i = t - l_g + 1}^t u_{gi} \leq 1- x_{g,t-l_g},  \quad  & \forall t \in [l_g + 1, T], \label{eq:IBM_2}
\end{IEEEeqnarray}
\item dispatch level limits:
\begin{equation}
	x_{gt} \underline{p}_{g} \leq 	{p}_{gt} \leq 	x_{gt} \overline{p}_{g}, \quad \forall  t \in [1,T]. \label{eq:generation_capacity}
\end{equation}
\end{itemize}
Ramping constraints are not considered until Section \ref{sec:ramping}. Therefore, a unit's feasible set is
\begin{IEEEeqnarray}{C}
		\mathcal{X}_g =  \lbrace \bm{p}_g \in \mathbb{R}^{T}, \bm{x}_g \in \lbrace 0,1 \rbrace^{T}, \bm{u}_g \in \lbrace 0,1 \rbrace^{T-1} \,|\, 
		    \text{\eqref{eq:state_transition}--\eqref{eq:IBM_2}}, \eqref{eq:generation_capacity} \rbrace.
\end{IEEEeqnarray}

The set of feasible binary decisions alone is:
 \begin{equation}
	\mathcal{D}_g = \lbrace \bm{x}_g \in \lbrace 0,1 \rbrace^{T}, \bm{u}_g \in \lbrace 0,1 \rbrace^{T-1} \,|\,  \text{\eqref{eq:state_transition}--\eqref{eq:IBM_2}}  \rbrace.
\end{equation}
The following trivial inequalities are valid for $\mathcal{D}_g$:
\begin{equation}
u_{gt} \geq 0, \quad  \forall t  \in [2,T]. \label{eq:IBM_trivial} 
\end{equation}
It is claimed in \cite{rajan2005minimum} that
\begin{IEEEeqnarray}{C}
	\text{conv} (\mathcal{D}_g) = \lbrace \bm{x}_g \in \mathbb{R}^{T}, \bm{u}_g \in \mathbb{R}^{T-1} \, |    \text{\eqref{eq:state_transition}--\eqref{eq:IBM_2}},  \eqref{eq:IBM_trivial} \rbrace. \quad \label{eq:IBM_3}
\end{IEEEeqnarray}
While this statement is true, the proof of it relies on a lemma that states that all the extreme points of the polytope defined in \eqref{eq:IBM_3} are integral, Lemma 2.9 in \cite{rajan2005minimum}. The proof of Lemma 2.9 given in \cite{rajan2005minimum} is flawed, however.\footnote{We thank Dr.\ Dane A.\ Schiro and Dr.\ Eugene Litvinov for pointing this out to us.} We identify the flaws and give our own proof in Appendix. 

We also note that the minimum up/down time constraints in \cite{rajan2005minimum} have been used in recent literature on tight formulations of the UCED problem, such as \cite{Ostrowski2012,Morales-Espana2013,Morales-Espana2016}.

The following theorem extends the result in \eqref{eq:IBM_3} to also include the dispatch decisions:

\begin{theorem}
	The result shown in \eqref{eq:IBM_3} implies that
\begin{IEEEeqnarray*}{C}
	\mathcal{C}_g =  \lbrace \bm{p}_g \in \mathbb{R}^{T}, \bm{x}_g \in \mathbb{R}^{T}, \bm{u}_g \in \mathbb{R}^{T-1} \, | \, 
	   \text{\eqref{eq:state_transition}--\eqref{eq:generation_capacity}}, \eqref{eq:IBM_trivial}\rbrace
\end{IEEEeqnarray*}
describes the convex hull of $\mathcal{X}_g$.\footnote{For $T = 2$ and $T = 3$, the convex hull of a more general $\mathcal{X}_g$ with ramping constraints has been characterized in \cite{Pan2015}. The result here does not consider ramping constraints, but holds for an arbitrary $T$.}
\label{thm:explicit_hull}
\end{theorem}
\begin{IEEEproof}
	Since all the inequalities that describe $\mathcal{C}_g$ are valid, it suffices to prove that all extreme points of $\mathcal{C}_g$ have binary values for commitment and start-up variables.
	
	Suppose that $(\hat{\bm{p}}_g, \hat{\bm{x}}_g, \hat{\bm{u}}_g)$ is an extreme point of $\mathcal{C}_g$. By definition, among constraints  \eqref{eq:state_transition}--\eqref{eq:generation_capacity}, \eqref{eq:IBM_trivial}, $3T-1$ linearly independent constraints are active at $(\hat{\bm{p}}_g, \hat{\bm{x}}_g, \hat{\bm{u}}_g)$. Only constraints \eqref{eq:generation_capacity} involve vector $\bm{p}_g$. Therefore, out of the $3T-1$ active constraints, $T$ of them must be of type \eqref{eq:generation_capacity}.
	
For the $T$ active constraints of type \eqref{eq:generation_capacity}, either $\hat{p}_{gt} = \hat{x}_{gt} \underline{p}_{g}$ or $\hat{p}_{gt} = \hat{x}_{gt} \overline{p}_{g}$. Let $\lbrace \underline{\mathcal{T}},\overline{\mathcal{T}} \rbrace$ be a partition of the set $\lbrace 1,\ldots, T \rbrace$, so that $\hat{p}_{gt} = \hat{x}_{gt} \underline{p}_{g}, \, \forall t \in \underline{\mathcal{T}}$, and that $\hat{p}_{gt} = \hat{x}_{gt} \overline{p}_{g}, \, \forall t \in \overline{\mathcal{T}}$. It is easy to show that these active constraints are linearly independent. 

More importantly, the projection of 
\begin{IEEEeqnarray}{l}
			\lbrace \bm{p}_g \in \mathbb{R}^{T}, \bm{x}_g \in \mathbb{R}^{T}, \bm{u}_g \in \mathbb{R}^{T-1} \, | \, 
 {p}_{gt} = {x}_{gt} \underline{p}_{g}, \, \forall t \in \underline{\mathcal{T}} , \quad
			  {p}_{gt} = {x}_{gt} \overline{p}_{g}, \, \forall t \in \overline{\mathcal{T}}\rbrace \label{eq:set_projection}
	\end{IEEEeqnarray}
onto the $({\bm{x}}_g, {\bm{u}}_g)$-space is the whole of $\mathbb{R}^{T} \times \mathbb{R}^{T-1}$. To see this, consider any point $({\bm{x}}_g', {\bm{u}}_g') \in \mathbb{R}^{T} \times \mathbb{R}^{T-1}$. Let ${p}_{gt}' = {x}_{gt}' \underline{p}_{g}, \, \forall t \in \underline{\mathcal{T}}$, and let ${p}_{gt}' = {x}_{gt}' \overline{p}_{g}, \, \forall t \in \overline{\mathcal{T}}$. By construction, the point $({\bm{p}}_g', {\bm{x}}_g', {\bm{u}}_g')$ is in set \eqref{eq:set_projection}.

	In addition to the $T$ active constraints of type \eqref{eq:generation_capacity}, $2T-1$ other linearly independent constraints must be active at $(\hat{\bm{p}}_g, \hat{\bm{x}}_g, \hat{\bm{u}}_g)$. These $2T-1$ constraints can only be of type \eqref{eq:state_transition}--\eqref{eq:IBM_2} or \eqref{eq:IBM_trivial}. Because the projection of the set \eqref{eq:set_projection} onto the $({\bm{x}}_g, {\bm{u}}_g)$-space is $\mathbb{R}^{T} \times \mathbb{R}^{T-1}$, and because the number of such active constraints equals the number of variables, the values of $(\hat{\bm{x}}_g, \hat{\bm{u}}_g)$ are completely determined by these $2T-1$ active constraints. Lemma 2.9 in \cite{rajan2005minimum} implies that any such $2T-1$ constraints lead to binary $(\hat{\bm{x}}_g, \hat{\bm{u}}_g)$. 
\end{IEEEproof}

\subsection{Characterization of the Convex Envelopes}

In \textbf{CHP-Primal}, each cost function $C_g(\cdot)$ is replaced by its convex envelope taken over the non-convex feasible set $\mathcal{X}_g$. When a unit has a constant marginal cost, $C_g(\cdot)$ is affine, and the convex envelope of $C_g(\cdot)$ has the same functional form as $C_g(\cdot)$ itself. 

When $C_g(\cdot)$ is not affine (piecewise linear or quadratic in $\bm{p}_g$), its convex envelope has a different functional form. We first discuss the convex quadratic case. 

Let the start-up and no-load cost of unit $g$ be $h_g$ and $c_g$, respectively. Define the following set:
\begin{IEEEeqnarray*}{r}
	\mathcal{X}_{gt} =  \lbrace p_{gt} \in \mathbb{R}, x_{gt} \in \lbrace 0,1 \rbrace, u_{gt} \in \lbrace 0,1 \rbrace \,|\, 
	x_{gt} \underline{p}_{g} \leq 	{p}_{gt} \leq 	x_{gt} \overline{p}_{g} \rbrace.
\end{IEEEeqnarray*}
Suppose the offer cost function $C_{g}: \mathcal{X}_g \rightarrow \mathbb{R}$ is defined by:
\begin{equation}
C_g(\bm{p}_g, \bm{x}_g, \bm{u}_g) = \sum_{t = 1}^T C_{gt}({p}_{gt}, {x}_{gt}, {u}_{gt}), \label{eq:cost_function_quad}
\end{equation}
where for each period $t$, $C_{gt}: \mathcal{X}_{gt} \rightarrow \mathbb{R}$ is a convex quadratic function defined by a \emph{single-period cost function}:
\begin{equation}
 C_{gt}({p}_{gt}, {x}_{gt}, {u}_{gt}) =  a_gp_{gt}^2 + b_gp_{gt} + c_gx_{gt} + h_gu_{gt}, \label{eq:cost_function_quad_single_period}
\end{equation}	
where we assume $a_g >0$. The convex envelope of $C_g$ is characterized by the following theorem.\footnote{Akt\"{u}rk et al. \cite{Akturk2009} describe a similar result that considers a more general polynomial function. However, our result considers a multi-period domain, and is more general in another direction. We thank Dr.\ Alper Atamt\"{u}rk for pointing this out to us.}

\begin{theorem}
\label{thm:convex_envelope_quadratic}
	The convex envelope of the quadratic cost function $C_g$ taken over $\mathcal{X}_g$ is
the function $C_{g, \mathcal{X}_g}^{**}: \text{\emph{conv}}(\mathcal{X}_g) \rightarrow \mathbb{R}$ defined by the following:
	\begin{equation*}
C_{g, \mathcal{X}_g}^{**}(\bm{p}_g, \bm{x}_g, \bm{u}_g) = 	\sum_{t = 1}^T C_{gt, \mathcal{X}_{gt}}^{**}({p}_{gt}, {x}_{gt}, {u}_{gt}),
	\end{equation*}
	where $C_{gt, \mathcal{X}_{gt}}^{**}: \text{\emph{conv}}(\mathcal{X}_{gt}) \rightarrow \mathbb{R}$ is defined by the following:
	\begin{IEEEeqnarray*}{c}
	C_{gt, \mathcal{X}_{gt}}^{**}({p}_{gt}, {x}_{gt}, {u}_{gt}) 
= \begin{cases}
		a_g \frac{p_{gt}^2}{x_{gt}} + b_gp_{gt} + c_gx_{gt} + h_gu_{gt}, &x_{gt}>0,\\
		0,& x_{gt} = 0.
	\end{cases}
	\end{IEEEeqnarray*}
\end{theorem}
\begin{IEEEproof}
To develop intuition for this proof, we first consider a restricted case where there is only a single period with the start-up variable fixed at zero. Such restriction allows us to plot the graph of the cost function and its convex envelope in a three-dimension space. We then complete the proof by considering the general multi-period case.\footnote{We thank Jian Sun from Tsinghua University for pointing out a missing step in a previous version of this proof.}

First, consider the single-period cost function $C_{gt}$ with $u_{gt}$ fixed at zero. Consider the value of the function:
\begin{equation}
C_{gt, \mathcal{X}_{gt}}^{**}({p}_{gt}, {x}_{gt}, 0) =
	\begin{cases}
		a_g \frac{p_{gt}^2}{x_{gt}} + b_gp_{gt} + c_gx_{gt}, &x_{gt}>0,\\
		0,& x_{gt} = 0,
	\end{cases}
	\label{eq:envelope_single_period}
	\end{equation}
as $({p}_{gt}, {x}_{gt})$ varies over
$\lbrace p_{gt} \in \mathbb{R}, x_{gt} \in [0,1] \,| \,	x_{gt} \underline{p}_{g} \leq 	{p}_{gt} \leq 	x_{gt} \overline{p}_{g}\rbrace$. 
When $x_{gt} \in \lbrace 0,1 \rbrace$, the graph of this function is the same as  $C_{gt}({p}_{gt}, {x}_{gt}, 0) $. 
The graph of this function at any point $({p}_{gt}, {x}_{gt})$ with ${x}_{gt} \in (0,1)$ is determined by the line segment connecting $(0,0,0)$ and $(\frac{p_{gt}}{x_{gt}}, 1, C_{gt}(\frac{p_{gt}}{x_{gt}}, 1, 0))$. This function is continuous and convex on $\lbrace p_{gt} \in \mathbb{R}, x_{gt} \in [0,1] \,| \,	x_{gt} \underline{p}_{g} \leq 	{p}_{gt} \leq 	x_{gt} \overline{p}_{g}\rbrace$, as can be verified by taking its Hessian in this domain. 

We prove by contradiction that among the convex under-estimators of  $C_{gt}({p}_{gt}, {x}_{gt}, 0)$ on the given domain, $C_{gt, \mathcal{X}_{gt}}^{**}({p}_{gt}, {x}_{gt}, 0)$ is the largest one. Suppose not, then there exists a convex under-estimator of $C_{gt}({p}_{gt}, {x}_{gt}, 0)$, denoted by $C'_{gt}({p}_{gt}, {x}_{gt}, 0)$, for which there exist a point $({p}'_{gt}, {x}'_{gt})$ with ${x}'_{gt} \in (0,1)$ such that $C'_{gt}({p}'_{gt}, {x}'_{gt}, 0) > C^{**}_{gt}({p}'_{gt}, {x}'_{gt}, 0)$. 

Now we focus on the graph of the functions, which exists in a space of dimension four. Consider the line interval in the graph space connecting $(0,0,0,0)$ and $(\frac{p'_{gt}}{x'_{gt}}, 1,0,C_{gt}(\frac{p'_{gt}}{x'_{gt}}, 1, 0))$. It is easy to verify that the graph of the function at $({p}'_{gt}, {x}'_{gt}, 0)$ lies on this line interval. We have $C'_{gt}({p}'_{gt}, {x}'_{gt}, 0) > C^{**}_{gt}({p}'_{gt}, {x}'_{gt}, 0) = 0 + {x}'_{gt}  C_{gt}(\frac{p'_{gt}}{x'_{gt}}, 1, 0)$, which implies that $C'_{gt}$ is not convex when restricted to this line. This contradicts the convexity of $C'_{gt}$, since a function is convex if and only if it is convex when restricted to any line that intersects its domain \cite[Chapter 3]{Boyd2004}. Therefore, $C_{gt, \mathcal{X}_{gt}}^{**}({p}_{gt}, {x}_{gt}, 0)$ is the convex envelope of $C_{gt}({p}_{gt}, {x}_{gt}, 0)$.

Fig. \ref{fig:convex_envelope} shows an example of $C_{gt}({p}_{gt}, {x}_{gt}, 0)$ and its convex envelope. Since $C_{gt}$ is affine in $u_{gt}$, $C_{gt, \mathcal{X}_{gt}}^{**}({p}_{gt}, {x}_{gt}, {u}_{gt})$ is the convex envelope of $C_{gt}({p}_{gt}, {x}_{gt}, {u}_{gt})$.

\begin{figure}
	\centering 
	\includegraphics[width=0.5\textwidth]{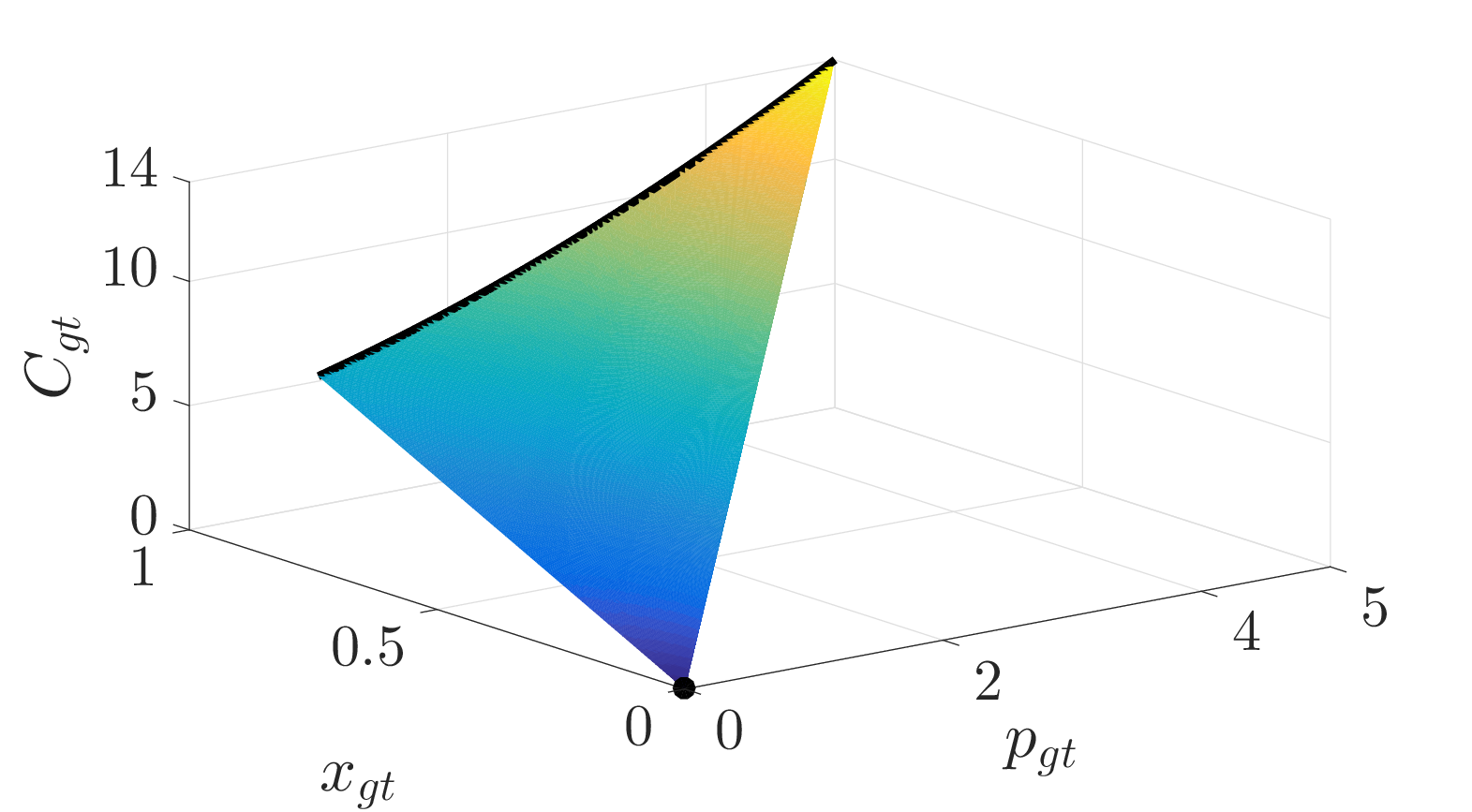}
	\caption{Convex envelope of a single-period cost function. The figure shows the graph of a single-period cost function $0.2p_{gt}^2 + p_{gt} + 4x_{gt}$ defined on $\lbrace x_{gt} \in \lbrace 0,1 \rbrace, p_{gt} \in \mathbb{R} \,| \, x_{gt} \leq p_{gt} \leq 5x_{gt}\rbrace$ (black dot and black curved line), together with its convex envelope (the colored surface).}\label{fig:convex_envelope}
\end{figure}
	
Next, we consider the general multi-period case. Suppose that $C_{g, \mathcal{X}_g}^{**}(\bm{p}_g, \bm{x}_g, \bm{u}_g)$ is not the largest convex under-estimator of $C_g(\bm{p}_g, \bm{x}_g, \bm{u}_g)$. Then, there exists a convex under-estimator $C^\dagger_g(\bm{p}_g, \bm{x}_g, \bm{u}_g)$ of $C_g(\bm{p}_g, \bm{x}_g, \bm{u}_g)$, for which there exists a point $(\bm{p}_g^{\dagger}, \bm{x}_g^{\dagger}, \bm{u}_g^{\dagger}) \in \text{conv}(\mathcal{X}_g)$ with $\bm{0} < \bm{x}_g^{\dagger} < \bm{1}$ such that $C^\dagger_g(\bm{p}_g^{\dagger}, \bm{x}_g^{\dagger}, \bm{u}_g^{\dagger}) > C_{g, \mathcal{X}_g}^{**}(\bm{p}_g^{\dagger}, \bm{x}_g^{\dagger}, \bm{u}_g^{\dagger})$.
	
	 Let us construct a point, $(\bm{p}_g^{\dagger\dagger}, \bm{x}_g^{\dagger\dagger}, \bm{u}_g^{\dagger\dagger})$ where
		\begin{IEEEeqnarray*}{ll}
	p_{gt}^{\dagger\dagger} = \begin{cases}
		\frac{p_{gt}^\dagger}{x_{gt}^\dagger} , &x_{gt}^\dagger>0,\\
		0,& x_{gt}^\dagger= 0,
	\end{cases}
	\end{IEEEeqnarray*}
	
			\begin{IEEEeqnarray*}{ll}
	x_{gt}^{\dagger\dagger} = \begin{cases}
		1, &x_{gt}^\dagger>0,\\
		0,& x_{gt}^\dagger= 0,
	\end{cases}
	\end{IEEEeqnarray*}	
	and 
		\begin{IEEEeqnarray*}{l}
	u_{gt}^{\dagger\dagger} = x_{gt}^{\dagger\dagger} - x_{g,t-1}^{\dagger\dagger}, \quad \forall t \in [2,T].
	\end{IEEEeqnarray*}	
We can see that $(\bm{p}_g^{\dagger\dagger}, \bm{x}_g^{\dagger\dagger}, \bm{u}_g^{\dagger\dagger}) \in \mathcal{X}_g$ because $\bm{x}_g^{\dagger\dagger}$ is integral, and that $\bm{u}_g^{\dagger\dagger}$ is constructed via the state-transition relationship between the startup and the commitment variables.

Because $C^\dagger_g$ is an under-estimator, 
\begin{IEEEeqnarray*}{c}
C^\dagger_g(\bm{p}_g^{\dagger\dagger}, \bm{x}_g^{\dagger\dagger}, \bm{u}_g^{\dagger\dagger}) \leq C_g(\bm{p}_g^{\dagger\dagger}, \bm{x}_g^{\dagger\dagger}, \bm{u}_g^{\dagger\dagger});
\end{IEEEeqnarray*}	
because $C^\dagger_g$ is a larger under-estimator than $C_{g, \mathcal{X}_g}^{**}$, we have 
\begin{IEEEeqnarray*}{c}
C^\dagger_g(\bm{p}_g^{\dagger\dagger}, \bm{x}_g^{\dagger\dagger}, \bm{u}_g^{\dagger\dagger}) \geq C_{g, \mathcal{X}_g}^{**}(\bm{p}_g^{\dagger\dagger}, \bm{x}_g^{\dagger\dagger}, \bm{u}_g^{\dagger\dagger}) = C_g(\bm{p}_g^{\dagger\dagger}, \bm{x}_g^{\dagger\dagger}, \bm{u}_g^{\dagger\dagger}).
\end{IEEEeqnarray*}	
The two equations above imply 
\begin{IEEEeqnarray*}{c}
C^\dagger_g(\bm{p}_g^{\dagger\dagger}, \bm{x}_g^{\dagger\dagger}, \bm{u}_g^{\dagger\dagger}) = C_{g, \mathcal{X}_g}^{**}(\bm{p}_g^{\dagger\dagger}, \bm{x}_g^{\dagger\dagger}, \bm{u}_g^{\dagger\dagger}) = C_g(\bm{p}_g^{\dagger\dagger}, \bm{x}_g^{\dagger\dagger}, \bm{u}_g^{\dagger\dagger}).
\end{IEEEeqnarray*}	
Similarly, we have 
\begin{IEEEeqnarray*}{c}
C^\dagger_g(\bm{0}, \bm{0}, \bm{0}) = C_{g, \mathcal{X}_g}^{**}(\bm{0}, \bm{0}, \bm{0}) = C_g(\bm{0}, \bm{0},\bm{0}) = 0.
\end{IEEEeqnarray*}	

Now let us focus on the graph of these convex under-estimators, which exist in a space of dimension $3T$ with the last dimension being the value of the function. Consider the line interval between $(\bm{0},\bm{0},\bm{0}, 0)$ and $(\bm{p}_g^{\dagger\dagger}, \bm{x}_g^{\dagger\dagger}, \bm{u}_g^{\dagger\dagger}, C_{g, \mathcal{X}_g}^{**}(\bm{p}_g^{\dagger\dagger}, \bm{x}_g^{\dagger\dagger}, \bm{u}_g^{\dagger\dagger}))$. It is easy to verify that the point $(\bm{p}_g^{\dagger}, \bm{x}_g^{\dagger}, \bm{u}_g^{\dagger}, C^\dagger_g(\bm{p}_g^{\dagger}, \bm{x}_g^{\dagger}, \bm{u}_g^{\dagger}) )$ lies on this line interval. 

The inequality  $C^\dagger_g(\bm{p}_g^{\dagger}, \bm{x}_g^{\dagger}, \bm{u}_g^{\dagger}) > C_{g, \mathcal{X}_g}^{**}(\bm{p}_g^{\dagger}, \bm{x}_g^{\dagger}, \bm{u}_g^{\dagger})$ implies that the point on the graph of $C^\dagger_g$, $(\bm{p}_g^{\dagger}, \bm{x}_g^{\dagger}, \bm{0}, C_{g, \mathcal{X}_g}^\dagger(\bm{p}_g^{\dagger}, \bm{x}_g^{\dagger}, \bm{0}))$, lies above the linear combination of the two points on the same graph, $(\bm{0},\bm{0},\bm{0}, 0)$ and
\begin{equation*}
	(\bm{p}_g^{\dagger\dagger}, \bm{x}_g^{\dagger\dagger}, \bm{u}_g^{\dagger\dagger}, C^\dagger_g(\bm{p}_g^{\dagger\dagger}, \bm{x}_g^{\dagger\dagger}, \bm{u}_g^{\dagger\dagger} )).
\end{equation*}
Since a function is convex if and only if it is convex when restricted to any line that intersects its domain, this in turn implies that $C^\dagger_g(\bm{p}_g, \bm{x}_g, \bm{u}_g)$ is not convex, a contradiction.

\end{IEEEproof}

Using the convex envelopes gives us a better lower bound than simply keeping the functional form of $C_{gt, \mathcal{X}_{gt}}$ and relaxing its domain. Given $g \in \mathcal{G}$, $t \in [1,T]$, and $p_{gt} > 0$, when $x_{gt}$ is binary, $C_{gt, \mathcal{X}_{gt}}^{**}({p}_{gt}, {x}_{gt}, 0) = C_{gt, \mathcal{X}_{gt}}({p}_{gt}, {x}_{gt}, 0)$; when $x_{gt}$ is fractional, we have $\frac{p_{gt}^2}{x_{gt}}> {p_{gt}^2}$.

We next consider convex piecewise linear cost functions. Suppose the interval $[\underline{p}_g, \overline{p}_g]$ is partitioned into $|\mathcal{K}|$ intervals:
\begin{equation}
	[\underline{p}_g, \overline{p}_g] = \bigcup_{k \in \mathcal{K}} \mathcal{I}_k,
\end{equation}
where $k$ is the index for the partitioned intervals.

Recalling that $h_g$ denotes the start-up cost of unit $g$, suppose at each period $t$, when $p_{gt} \in \mathcal{I}_k$, the operating cost (excluding start-up cost and no-load cost) is
\begin{equation}
	\tilde{C}_{gt}({p}_{gt}, 1, 0) =  a_{gk}p_{gt} + b_{gk}. \label{eq:cost_function_lin_single_period}
\end{equation}
Introducing an auxiliary variable $s_{gt}\in \mathbb{R}_+$, the single-period piecewise linear cost function can be implemented as
\begin{equation}
	\tilde{C}_{gt}({p}_{gt}, {x}_{gt}, {u}_{gt}) = s_{gt} + c_gx_{gt} + h_gu_{gt}, \label{eq:PWL1}
\end{equation}
\begin{equation}
	s_{gt} \geq  a_{gk}p_{gt} + b_{gk}, \forall k \in \mathcal{K}. \label{eq:PWL2}
\end{equation}

The convex envelope of a convex piecewise linear cost function taken over $\mathcal{X}_g$ is also piecewise linear. The proof is similar to that for Theorem \ref{thm:convex_envelope_quadratic}.
\begin{theorem}
\label{thm:convex_envelope_lin}
	The convex envelope of the convex piecewise linear cost function $\tilde{C}_{gt}$ taken over $\mathcal{X}_g$ is the function $\tilde{C}_{g, \mathcal{X}_g}^{**}: \text{\emph{conv}}(\mathcal{X}_g) \rightarrow \mathbb{R}$ defined by:
	\begin{IEEEeqnarray*}{c}
 \tilde{C}_{g, \mathcal{X}_g}^{**}(\bm{p}_g, \bm{x}_g, \bm{u}_g) 
=   \sum_{t = 1}^T
\begin{cases}
		 a_{gk}p_{gt} + (c_{g}+b_{gk})x_{gt} + h_gu_{gt}, &\text{\emph{if}} \,\,  \frac{p_{gt}}{x_{gt}} \in \mathcal{I}_k, \\
		 0,& \text{\emph{if}} \,\,  x_{gt} = 0.
	\end{cases}
	\end{IEEEeqnarray*}
\end{theorem}

Typically, $b_{gk}$ is non-positive. Therefore, similar to the quadratic case, using the convex envelopes gives us a better lower bound than the implementation shown in \eqref{eq:PWL1} and \eqref{eq:PWL2}.

\subsection{Reformulation and Polynomial-Time Solution}

Theorems \ref{thm:primal_formulation}--\ref{thm:convex_envelope_lin} imply that, if each unit faces only generation limits and minimum up/down constraints, exact convex hull prices can be determined by solving \textbf{CHP-Primal} with the convex hulls and convex envelopes explicitly described. All variables are continuous in \textbf{CHP-Primal}. 

When the cost functions are quadratic, non-linearity of the convex envelope comes from the quadratic-over-linear terms $a_g \frac{p_{gt}^2}{x_{gt}}$ which are known to be convex. Moreover, we can move these terms from the objective into constraints and cast \textbf{CHP-Primal} as a second-order cone program (SOCP). 

For each $g$ and $t$, we replace $a_g \frac{p_{gt}^2}{x_{gt}}$ by a new variable $s_{gt} \in \mathbb{R}_+$ and introduce the following constraint: 
\begin{equation}
s_{gt} x_{gt} \geq a_g p_{gt}^2. \label{eq:quad_over_linear}
\end{equation}
Since we are minimizing, when $x_{gt} = 0$, the optimal value for $s_{gt}$ is zero, which is consistent with the convex envelope. For $x_{gt} \geq 0$ and $s_{gt} \geq 0$, constraint \eqref{eq:quad_over_linear} is equivalent to
\begin{equation}
\lVert (2\sqrt{a_g}p_{gt}, x_{gt} - s_{gt})\rVert_2 \leq  x_{gt} + s_{gt},
\end{equation}
which is a second-order cone constraint \cite{Lobo1998}. With this reformulation technique, \textbf{CHP-Primal} can be cast as an SOCP, which can be solved in polynomial time using off-the-shelf interior-point solvers, e.g. GUROBI \cite{gurobi}. 

In the case where the cost functions are piecewise linear, the convex envelope of the cost function is convex piecewise linear. The resulting \textbf{CHP-Primal} is a linear program (LP). 

Since the number of constraints in our explicit formulation is polynomial in $T$ and $|\mathcal{G}|$, the convex hull pricing problem can be solved as a convex program in polynomial time in both cases. Note that in an optimal solution to \textbf{CHP-Primal}, the commitment and start-up variables can be fractional. In convex hull pricing, we focus on the optimality of dual variables. The ISO's commitment and dispatch decisions are still determined by the UCED problem.

\section{Extensions}
\label{sec:extensions}

\subsection{Transmission and Other Linear System-Wide Constraints} 

To determine locational convex hull prices, we consider a linear approximation to the transmission constraints and augment \textbf{CHP-Primal} with angle-eliminated transmission constraints in terms of the shift factors. The locational prices can be derived as a function of the dual variables associated with the supply-demand balance constraints and the transmission constraints, as in locational marginal pricing \cite[Chapter 8.11]{wood2012power}. Other linear system-wide constraints (e.g., contingency constraints, constraints that approximate loss in the transmission system) can be treated in a similar fashion. Theorem \ref{thm:primal_formulation} does not apply to nonlinear system-wide constraints \cite{Lemarechal2001}.

Note that in the presence of system-wide constraints that do not necessarily hold as equalities at a welfare-maximizing solution, such as the transmission constraints, the gap between the UCED problem and its dual includes not only the total lost opportunity cost of the units, but also another type of uplift that addresses the ISO's revenue insufficiency \cite{Cadwalader2010, Schiro2015}. 

\subsection{Ancillary Services} 
In markets where energy and ancillary services are co-optimized, a set of variables are introduced to represent the ancillary services provided by market participants. We use spinning reserve as an example. 

Let vector $\bm{r}_g$ denote the amount of spinning reserve provided by generator $g$ in each time period. Let $\underline{r}_{g}$ and $\overline{r}_{g}$ be the lower and upper limits on upward spinning reserve. We include the following constraints: 
\begin{equation}
	x_{gt} \underline{p}_{g} \leq {p}_{gt}, \quad \forall  t \in [1,T],\label{eq:reserve_1}
\end{equation}
\begin{equation}
	{p}_{gt} + {r}_{gt}  \leq 	x_{gt} \overline{p}_{g}, \quad \forall  t \in [1,T],\label{eq:reserve_11}
	\end{equation}
\begin{equation}
	x_{gt} \underline{r}_{g} \leq {r}_{gt} \leq 	x_{gt} \overline{r}_{g}, \quad \forall  t \in [1,T]  \label{eq:reserve_2}.
\end{equation}
The feasible set for each unit is redefined to be 
\begin{IEEEeqnarray*}{C}
		\mathcal{X}_g =  \lbrace \bm{p}_g \in \mathbb{R}^{T}, \bm{r}_g \in \mathbb{R}^{T},  \bm{x}_g \in \lbrace 0,1 \rbrace^{T}, \bm{u}_g \in \lbrace 0,1 \rbrace^{T-1} \,|\,  
		\eqref{eq:state_transition} \text{--}\eqref{eq:IBM_2},  \eqref{eq:reserve_1}\text{--}\eqref{eq:reserve_2} \rbrace.
\end{IEEEeqnarray*}

We can show that 
\begin{IEEEeqnarray*}{C}
	\text{conv}(\mathcal{X}_g ) =  \lbrace \bm{p}_g \in \mathbb{R}^{T}, \bm{r}_g \in \mathbb{R}^{T}, \bm{x}_g \in \mathbb{R}^{T}, \bm{u}_g \in \mathbb{R}^{T-1} \, | \,  
	  \eqref{eq:state_transition} \text{--}\eqref{eq:IBM_2}, \eqref{eq:IBM_trivial}, \eqref{eq:reserve_1}\text{--}\eqref{eq:reserve_2}\rbrace. 
\end{IEEEeqnarray*}
The proof is similar to that for Theorem \ref{thm:explicit_hull}. 

Since the convex envelope of a convex function taken over a convex domain is the convex function itself, and since the additional constraints \eqref{eq:reserve_1} and \eqref{eq:reserve_2} define a convex feasible set, introducing ancillary services does not alter the convex envelope. Since we explicitly characterize both the convex hulls and convex envelopes, we can obtain exact convex hull prices using \textbf{CHP-Primal} augmented with \eqref{eq:reserve_1}, \eqref{eq:reserve_2}, and system-wide constraints for ancillary services.

\subsection{Ramping Constraints}
\label{sec:ramping}

 Let $\overline{v}_g$ denote unit $g$'s start-up/shut-down ramp rate limit, and let $v_g$ be unit $g$'s ramp-up/down rate when committed. Ramping constraints are a set of private constraints that limit the increase or decrease of power output from one time period to the next \cite{Pan2015}:
\begin{equation}
	p_{gt} - p_{g, t-1} \leq v_g x_{g, t-1} + \overline{v}_g (1 - x_{g, t-1}), \, \forall  t \in [2,T], \label{eq:ramping_1} 
\end{equation}
\begin{equation}
	p_{g, t-1} - p_{gt}  \leq v_g x_{gt} + \overline{v}_g (1 - x_{gt}), \, \forall  t \in [2,T]. \label{eq:ramping_2}
\end{equation}
We redefine the feasible set for each unit to be 
\begin{IEEEeqnarray}{C}
		\mathcal{X}_g =  \lbrace \bm{p}_g \in \mathbb{R}^{T}, \bm{x}_g \in \lbrace 0,1 \rbrace^{T}, \bm{u}_g \in \lbrace 0,1 \rbrace^{T-1} \,|\,  
		\eqref{eq:state_transition} \text{--}\eqref{eq:IBM_2}, \eqref{eq:generation_capacity}, \eqref{eq:ramping_1}, \eqref{eq:ramping_2}\rbrace.
\end{IEEEeqnarray}

Ramping constraints define a convex feasible set, and thus do not change the convex envelope of the cost function. They complicate the convex hulls, however. When these time-coupled constraints are included in the definition of $\mathcal{X}_g$, equations \eqref{eq:state_transition}\text{--}\eqref{eq:IBM_2}, \eqref{eq:IBM_trivial}, \eqref{eq:generation_capacity}, together with ramping constraints \eqref{eq:ramping_1} and \eqref{eq:ramping_2} themselves, do not completely characterize $\text{conv}(\mathcal{X}_g )$. Additional valid inequalities are needed to describe the convex hulls. The number of valid inequalities needed is in general exponential in $T$ \cite{Damc-Kurt2015}. 

Explicit descriptions of $\text{conv}(\mathcal{X}_g)$ for the case where $T = 2$ and $T = 3$ are shown in \cite{Pan2015}. More importantly, valid inequalities in these descriptions can be applied to any two or three consecutive time periods to tighten the approximation of $\text{conv}(\mathcal{X}_g)$ for $T>3$. 

When considering ramping constraints, we solve an approximation of \textbf{CHP-Primal} that includes the above-mentioned valid constraints for $T = 2$ and $T=3$. Our description of $\text{conv}(\mathcal{X}_g)$ is not exact. Consequently, our approximation provides a lower bound for \textbf{CHP-Primal}. The gap between the approximated \textbf{CHP-Primal} and the UCED problem gives an upper bound for the duality gap between the UCED problem and its Lagrangian dual problem.

A UCED problem only represents averaged ramping over the length of a time period. When the time resolution is one hour, sub-hourly ramping in opposite directions may cancel out. Therefore, we believe that our approximation is close, especially for a day-ahead market in which ramping constraints are less likely to be binding compared to a real-time market. Note also that a time-decoupled pricing problem is used in MISO's single-hour approximation of convex hull pricing \cite{Wang2016}. This pricing problem includes ramping constraints, but does not capture the time-coupling implications of ramping.


\subsection{Minimization of Uplift Payments to a Subset of the Participating Units}

\label{sec:pCHP}

In certain electricity markets, only a subset of the participating units $\mathcal{G}$ can receive uplift payments. For example, it may be the case that only units dispatched to a strictly positive generation level are qualified for uplift payments; that is, units are not paid for merely participating in the market. In this case, the duality gap that convex hull pricing minimizes includes terms that might not end up being paid.

Let the set of units that are qualified to receive uplift payments be $\mathcal{G}' \subseteq \mathcal{G}$. If the qualifications can be determined prior to computing the prices, we can solve \textbf{CHP-Primal} with $\mathcal{G}$ replaced by $\mathcal{G}'$. The duality gap that the resulting prices minimize includes only uplift payments to units in $\mathcal{G}'$.

\subsection{Committing to Prices}
In convex hull pricing, prices that are coupled across multiple time periods as a whole minimize uplift payments over the specified horizon $T$ of the underlying UCED problem. A subset of these prices does not necessarily minimize uplift payments over any shorter time horizon that is a subset of $T$. In a day-ahead market, the whole set of 24 hourly prices is calculated and posted at once, so that the coupling would not be problematic insofar as the ISO commits to buying and selling at these prices. 

In contrast, look-ahead real-time markets are operated on a rolling basis where only the commitment, dispatch, and price calculated for the upcoming interval are implemented. Therefore, the coupling inherent across a \emph{single} look-ahead dispatch may not be represented appropriately by the sequence of convex hull prices, each of which corresponds to the upcoming interval in each successive look-ahead dispatch.

To make prices consistent across successive look-ahead dispatches, \cite{Hogan2016a} suggests that the pricing model represent past intervals in the convex hull pricing model, keep the commitment and dispatch decisions in the past as variables, but constrain prices in the past intervals to be equal to the realized prices. A simple way to achieve this in \textbf{CHP-Primal} is to add in each past interval a fictitious power source/sink to the system with infinite generation capacity, infinite consumption capacity, and a constant marginal cost/willingness to pay equal to the realized price. This power source/sink constrains past prices to be equal to the realized prices, and maintains the coupling between realized and upcoming prices to minimize uplift payments given the realized prices. In the transmission-constrained case, we can add in each past interval a fictitious power source/sink to the slack bus with a constant marginal cost equal to the realized price at the slack bus, and dualize transmission constraints that correspond to congested lines with a penalty equal to their realized optimal dual variable.

\section{Numerical Results}
\label{sec:results}

We implement \textbf{CHP-Primal} on a personal computer with a 2.2-GHz quad-core CPU and 16 GB of RAM. The optimization problems are modeled in CVX \cite{cvx} and solved with GUROBI 6.5 \cite{gurobi}. We consider four examples from the literature. The time resolution in all examples is one hour.

\subsection{Example 1}
We consider an example from \cite{Schiro2015} in which two units (including a block-loaded one) serve $35$ MW of load in a single period. We modify the original example by including a start-up cost for each unit. Table \ref{table:ex1_offer} specifies each unit's offers. Both units are assumed to be off initially. The optimal (and the only feasible) solution to the ISO's UCED problem is for unit 1 to generate $35$ MW and for unit 2 to stay offline.

Unit 1 is marginal\footnote{A marginal unit has an optimal dispatch level strictly between its maximum and minimum power output.} and sets the LMP. Seeing the LMP, unit 2's profit-maximizing decision is to go online and generate $50$ MW, which would result in a profit of \$$1900$. Therefore, based on the definition of $U_g$ in \eqref{eq:individual_uplift}, unit 2 has a lost opportunity cost of \$$1900$. The start-up cost of unit 1 is not covered by LMP. An uplift payment of \$$100$ is needed to make unit 1 whole (guarantee a non-negative profit). 

\textbf{CHP-Primal} for this example is an LP and gives the exact convex hull price (CHP). The resulting uplift payments are lower than those under LMP but still quite large (Table \ref{table:ex1_uplift}).


\renewcommand\arraystretch{1.3}
\begin{table}[!t]
\caption{Supply Offers in Example 1}
\label{table:ex1_offer}
\centering
\begin{tabular}{c c c c c c}
\firsthline
\multirow{2}{*}{\, Unit \,} & Start-up & No-load & Energy & $\underline{p}_g$ &$\overline{p}_g $\\
& \$ & \$ & \$/MWh & MW & MW \\
\hline
1 & 100 & 0 & 50 & 10 & 50\\
2 & 100 & 0 & 10 & 50 & 50 \\
\hline
\end{tabular}
\end{table}

\begin{table}[!t]
\caption{Comparison of Different Pricing Schemes for Example 1}
\label{table:ex1_uplift}
\centering
\begin{tabular}{c c c c c c}
\firsthline
\multirow{2}{*}{Pricing Scheme} & $\pi$ & $U_1$  & $U_2$ \\
 & \$/MWh & \$ &\$ \\
\hline
LMP & $50$ & $100$ & $1900$\\
CHP & $12$ & $1430$ & $0$\\
CHPq & $52$ & $30$ & - \\
\hline
\end{tabular}
\end{table}

Suppose that only units dispatched to a strictly positive generation level are qualified for uplift payments. In this situation, unit 2 does not receive any compensation for its lost opportunity cost. As suggested in Section \ref{sec:pCHP}, to minimize the uplift payment to the qualified units, we can instead solve \textbf{CHP-Primal} with unit 2 excluded. We refer to this pricing scheme as CHPq (CHP for qualified units). Table \ref{table:ex1_uplift} shows that, if only unit 1 is qualified for compensation, CHPq results in a lower uplift payment (\$$30$) than CHP (\$$1430$).

\subsection{Example 2}
We investigate a three-period two-unit example from \cite{MISO2010} with ramping constraints but without startup costs. Table \ref{table:ex2_offer} shows the supply offers and ramp rate limits. All units are assumed to be off initially. Table \ref{table:ex2_dispatch} presents the optimal commitment and dispatch decisions as well as the demand in each period. Ramping constraints require unit 2 to commit at $t = 2$ so that it can ramp up to the generation level needed at $t = 3$. Table \ref{table:ex2_uplift} displays energy prices and uplift payments under different pricing schemes. 

\begin{table}[!t]
\caption{Units in Example 2}
\label{table:ex2_offer}
\centering
\begin{tabular}{c c c c c c}
\firsthline
\multirow{2}{*}{\, Unit \,} & No-load & Energy & $\underline{p}_g$ &$\overline{p}_g $ & Ramp Rate\\
&  \$ & \$/MWh & MW & MW & MW/hr \\
\hline
1 & 0 & 60 & 0 & 100 & 120\\
2 & 600 & 56 & 0 & 100 & 60\\
\hline
\end{tabular}
\end{table}

Since unit 1 is the marginal unit in all three periods, the LMPs are set by unit 1 at \$$60$/MWh.  The payment based on LMPs covers all of unit 1 costs. An uplift payment of \$$560$ is needed to ``make unit 2 whole''.

We can approximate $\text{conv}(\mathcal{X}_g )$ with constraints \eqref{eq:state_transition}\text{--}\eqref{eq:generation_capacity}, \eqref{eq:IBM_trivial}, and ramping constraints. We refer to this pricing method as aCHP1 (approximate CHP). In aCHP2, we augment our formulation with valid inequalities describing $\text{conv}(\mathcal{X}_g)$ with $T=2$. Finally, using the description of $\text{conv}(\mathcal{X}_g )$ for $T=3$, we formulate the convex hulls in this three-period example, which results in the exact convex hull prices.

\begin{table}[!t]
\caption{Optimal Commitment and Dispatch for Example 2}
\label{table:ex2_dispatch}
\centering
\begin{tabular}{c c c c c c}
\firsthline
\multirow{2}{*}{\, $t$ \,} & $d_t$ & $x_{1,t}$ & $p_{1,t}$ & $x_{2,t}$ &$p_{2,t}$ \\
&  MW & & MW && MW \\
\hline
1 & 70 & 1 & 70 & 0& 0\\
2 & 100& 1 & 40 & 1& 60\\
3& 170 &1 & 70&1 & 100\\
\hline
\end{tabular}
\end{table}

\begin{table}[!t]
\caption{Comparison of Different Pricing Schemes for Example 2}
\label{table:ex2_uplift}
\centering
\begin{tabular}{c c c c c c c}
\firsthline
\multirow{2}{*}{Pricing Scheme} & $\pi_1$ & $\pi_2$ & $\pi_3$ & $U_1$  & $U_2$ \\
 & \$/MWh  & \$/MWh   & \$/MWh  & \$ &\$ \\
\hline
LMP & $60$ & $60$ & $60$ & $0$ & $560$\\
aCHP1 & $60$ & $60$ & $64$ & $120$ & $160$\\
aCHP2 & $60$ & $60$ & $65.6$ & $168$ & $0$ \\
CHP & $60$ & $60$ & $65.6$ & $168$ & $0$ \\
\hline
\end{tabular}
\end{table}

Table \ref{table:ex2_uplift} shows the energy prices and uplift payments under different pricing schemes. For this example, as the approximation of $\text{conv}(\mathcal{X}_g )$ becomes more accurate, the energy price at $t = 3$ increases. Roughly speaking, unit 1 has an increasing incentive to generate more than the ISO's optimal dispatch, increasing its lost opportunity cost, but unit 2's no-load costs can be better covered, decreasing its lost opportunity cost. The net effect is a decrease in total uplift as $\pi_3$ increases.

Note that the prices resulting from aCHP2 happen to equal the exact convex hull prices. This result implies that the valid inequalities for $T = 3$ are ``non-binding''. The approximation of $\text{conv}(\mathcal{X}_g)$ in aCHP2 is accurate enough to yield the exact convex hull prices for this example.

\subsection{Example 3}
We consider a 24-period 32-unit example from \cite{Wang2013a}. The cost functions for the units are linear. There are no ramping or transmission constraints. \textbf{CHP-Primal} is an LP through which the exact convex hull prices can be obtained. \textbf{CHP-Primal} solves in $0.02$ seconds, resulting in the same convex hull prices as reported in \cite{Wang2013a}. The duality gap between the UCED problem and its Lagrangian dual problem is \$$ 1\,148$, which equals the total lost opportunity cost. 

We also implement a standard sub-gradient method to solve the Lagrangian dual problem in the dual space. We adopt the step length update rule (c) shown in \cite[Theorem 10.4]{wolsey1998integer}. We implement the dual updates in MATLAB, and solve the inner-level integer programs with GUROBI. Since the standard sub-gradient method does not have a non-heuristic stopping criterion, we terminate the algorithm after 550 iterations when the objective function shows no improvement. The resulting objective function value is $0.88\%$ sub-optimal with respect to the exact dual maximum we obtained from \textbf{CHP-Primal}, and the total computational time is $37.3$s.
 
 \subsection{Example 4}
 
We consider a 96-period 76-unit 8-bus example that is based on structural attributes and data from ISO New England \cite{Krishnamurthy2016}. We consider the start-up costs, no-load costs, minimum up/down time constraints, and ramping constraints for the generation units. The cost functions are quadratic. We use Scenario 1 of the 90 load scenarios provided in \cite{Krishnamurthy2016}. Minimum generation levels for the units are not specified in the original data, so we let $\underline{p} = 0.8 \overline{p}$ for each nuclear plant and $\underline{p} = 0.6 \overline{p}$ for each coal-fired unit. The units' initial statuses are not provided. We solve a single-period UCED problem to obtain the optimal commitment and dispatch decisions for period 1. We use these optimal decisions as the units' initial statuses, and assume that the units have been on/off for sufficiently long time so that the minimum up/down time constraints are not initially binding. The flow limits of the 12 transmission lines in the system are not defined in the original data. Therefore, we first investigate the case without transmission constraints (Case 1). We then set a limit of $2100$ MW on the flow over each transmission line (Case 2).

For each case, we first solve the UCED problem and obtain the LMPs. We then determine convex hull prices using two methods: an approximation of \textbf{CHP-Primal} and the single-period approximation proposed in \cite{Wang2016}. Because of the ramping constraints, we approximate $\text{conv}(\mathcal{X}_g)$ using constraints \eqref{eq:state_transition}\text{--}\eqref{eq:generation_capacity}, \eqref{eq:IBM_trivial}, along with valid inequalities that completely characterize these convex hulls for $T = 2$ and $T=3$. We use the convex envelope described in Theorem \ref{thm:convex_envelope_quadratic} and solve the primal formulation as an SOCP. When solving the UCED problems, we include above-mentioned valid inequalities a priori, and set MIPgap to $0.01\%$.

Table \ref{table:ex4_UCED} shows the results for the UCED problem and the approximated \textbf{CHP-Primal}, as well as the relative gap between these two problems. The approximated \textbf{CHP-Primal} solves in polynomial time with respect to the number of constraints. However, if we were to solve the Lagrangian dual problem in the dual space for Case 2, there would be 2400 dual variables. Such a large number of dual variables due to transmission constraints creates difficulties for non-smooth optimization methods\cite{Wang2013b}.

The relative gap between approximated \textbf{CHP-Primal} and the UCED problem (called CHP gap hereafter) is $0.06\%$ for Case 1 and $0.10\%$ for Case 2. This small CHP gap bounds two other gaps from above. First, the duality gap between the UCED problem and its Lagrangian dual problem can only be smaller than the CHP gap. This verifies the theoretical result shown in \cite{Bertsekas1982}, which states that the relative duality gap of the UCED problem and its Lagrangian dual approaches zero as the number of heterogeneous generators approaches infinity.  Second, the approximation error (the gap between the conceptual \textbf{CHP-Primal} and our approximation) is bounded from above by the CHP gap.

Table \ref{table:ex4_uplift} compares the total uplift payment under the three pricing schemes. We only consider units' lost opportunity costs. In the single-period approximation, start-up and no-load costs are considered only for fast-start units. We classify a unit with a minimum up/down time of one hour as a fast-start unit, and 18 units fall into this category. We allocate start-up costs to peak usage hours. In both cases, each single-period approximation solves in much less than a second. The convex hull prices derived from the proposed method result in the least uplift payment in both cases.

\begin{table}[!t]
\caption{UCED and approximated CHP-Primal for Example 4}
\label{table:ex4_UCED}
\centering
\begin{tabular}{l  r r}
\firsthline
& Case 1 & Case 2 \\
\hline
UCED Obj.(\$) & $41\,972\,688$ & $42\,303\,772$ \\
UCED CPU Time (s)& $87.42$ &$112.02$\\
\textbf{CHP-Primal} Obj. (\$) &$41\,946\,298$& $42\,259\,962$   \\
\textbf{CHP-Primal} CPU Time (s) & $6.22$ & $13.04$  \\
Gap (\%) &$0.063$ & $0.104$\\
\hline
\end{tabular}
\end{table}

\begin{table}[!t]
\caption{Total Lost Opportunity Cost (\$) for Example 4}
\label{table:ex4_uplift}
\centering
\begin{tabular}{r M{0.8cm} M{2.3cm} M{2.7cm}}
\firsthline
& LMP & Primal Formulation for CHP & Single-Period Approximation of CHP \\
\hline
Case 1& $183\,473$ & $33\,965$  & $96\,938$\\
Case 2& $329\,032$ &$40\,863$ & $177\,391$\\
\hline
\end{tabular}
\end{table}

 \section{Conclusions}
 \label{sec:conclusions}
 
This paper has proposed a polynomially-solvable primal formulation for the Lagrangian dual problem of the unit commitment and economic dispatch (UCED) problem. This primal formulation explicitly describes the convex hull of each unit's feasible set and the convex envelope of each unit's cost function. We show that exact convex hull prices can be obtained in the absence of ramping constraints, and that exactness is preserved when we consider ancillary services or any linear system-wide constraints. A tractable approximation applies when ramping constraints are considered. 
 
We cast our formulation as a second-order cone program if the cost functions are quadratic, and as a linear program if the cost functions are affine or piecewise linear. Convex hull prices are thereby determined in a robust and scalable manner. A 96-period 76-unit transmission-constrained example solves in less than fifteen seconds on a personal computer. This example shows that prices obtained through our formulation further reduce uplift payments compared to a single-period approximation of convex hull pricing.
 
The results of our paper have important applications beyond pricing. The convex envelopes in our paper can be used to tighten the formulation of the UCED problem (e.g., the ones proposed in \cite{Ostrowski2012} and \cite{Morales-Espana2013}). Our primal formulation is also a tight convex relaxation of the UCED problem. Because of computational complexity, many planning models for the generation and/or transmission system currently use only an embedded economic dispatch model for system operation, rather than a more realistic UCED model. The proposed primal formulation should provide a better and tractable approximation for system operation.

\section{Acknowledgement}
The authors would like to thank Dr.\ Paul Gribik, Dr.\ Eugene Litvinov, Dr.\ Dane A. Schiro, and Dr.\ R.\ Kevin Wood for discussions and valuable suggestions.

\newpage
\section*{Appendix: A Proof of Lemma 2.9 in \cite{rajan2005minimum}}

Over a time horizon $T \in \mathbb{Z}_{++}$, consider a generating unit with minimum up time of $L \in \mathbb{Z}_{+}$ dispatch intervals and minimum down time of $l \in \mathbb{Z}_{+}$ dispatch intervals. Let $\bm{x}$ (denoted by $\bm{u}$ in \cite{rajan2005minimum}) be the commitment vector and $\bm{u}$ (denoted by $\bm{v}$ in \cite{rajan2005minimum}) be the start-up vector.

The commitment polytope of a unit $D_T(L,l)$ is defined to be 
\begin{IEEEeqnarray}{rll}
D_T(L,l) = \lbrace \bm{x} \in \mathbb{R}^T, \bm{u} \in \mathbb{R}^{T - 1} \, | \, & \sum_{i = t - L + 1}^t u_i \leq x_t,  & \forall t \in [L + 1, T], \\
& \sum_{i = t - l + 1}^t u_i \leq 1 - x_{t - l},  \quad & \forall t \in [l + 1, T], \\
& u_t \geq x_t - x_{t - 1}, & \forall t \in [2, T], \\
& u_t \geq 0, & \forall t \in [2, T] \rbrace.
\end{IEEEeqnarray}

To simplify our proof, let $\bm{w}$ be the shut-down vector, so that 
\begin{equation}
w_t = u_t + x_{t - 1} - x_t, \forall t \in [2,T]. \label{eq:shutdown}
\end{equation}

Lemma 2.9 in  \cite{rajan2005minimum}:
\begin{lemma}
\leavevmode
Let $(\overline{\bm{x}}, \overline{\bm{u}}) \in D_T(L,l)$. Then there exist integral points $a^s \in D_T(L,l), s\in \mathcal{S}$, and $\lambda_s \in \mathbb{R}_+, s \in \mathcal{S}$ such that
\renewcommand{\theenumi}{(\roman{enumi})}%
\begin{enumerate}
\item $\overline{\bm{x}} = \sum_{s \in \mathcal{S}} \lambda_s \bm{x}(a^s)$,  $\overline{\bm{u}} = \sum_{s \in \mathcal{S}} \lambda_s \bm{u}(a^s)$, and $ \sum_{s \in \mathcal{S}} \lambda_s = 1$;
\item let $\mathcal{S}_t^d$ be the set of all points that have been shut down at $t$, then $\forall t \in [2,T]$, $\overline{w}_t = \sum_{s \in \mathcal{S}_t^d} \lambda_s$;
\item let $\mathcal{S}_t^u$ be the set of all points that have been started up at $t$, then $\forall t \in [2,T]$, $\overline{u}_t = \sum_{s \in \mathcal{S}_t^u} \lambda_s$,
\end{enumerate}
where $\bm{u}(a^s)$ is the $\bm{u}$ vector corresponding to $a^s$.
\end{lemma}

This lemma states that any point in $D_T(L,l)$ can be written as a convex combination of a set of integral points in this polytope. This implies that every extreme point of $D_T(L,l)$ is integral.

The proof of this lemma provided in \cite{rajan2005minimum} has the following flaws:
\begin{itemize}
\item the proof by induction uses a base case of $D_2(1,1)$ and an inductive step that shows the statements are true for $D_{T}(L,l)$ if they are true for $D_{T-1}(L,l)$. Multiple natural numbers are varying, and induction must be applied to $T$, $L$ and $l$;
\item in the inductive step, the authors consider  ``$\overline{u}_t$ of the integral points'', which is not well-defined, since a fraction of a natural number can be fractional;
\item the determination of $\lambda_s$ is not specified.
\end{itemize}

We state the lemma in a slightly different way and give a proof of the new lemma.

\begin{lemma}
\leavevmode
For all $T \in \mathbb{Z}_{++}\backslash\lbrace 1\rbrace$, $\forall L,l \in \mathbb{Z}_{+}$ such that $L \leq T-1$ and $l \leq T-1$, $\forall (\overline{\bm{x}}, \overline{\bm{u}}) \in D_T(L,l)$, there exists a set of integral points $\lbrace (\hat{\bm{x}}^s, \hat{\bm{u}}^s), s \in \mathcal{S} = \lbrace 1, 2, \ldots, |\mathcal{S}|\rbrace \rbrace \subset D_T(L,l) $, and $\lbrace \lambda^s \in \mathbb{R}_+, s \in \mathcal{S} \rbrace$ such that
\renewcommand{\theenumi}{(\roman{enumi})}%
\begin{enumerate}
\item $\overline{\bm{x}} = \sum_{s \in \mathcal{S}} \lambda^s \hat{\bm{x}}^s$,  $\overline{\bm{u}} = \sum_{s \in \mathcal{S}} \lambda^s \hat{\bm{u}}^s$, and $ \sum_{s \in \mathcal{S}} \lambda^s = 1$;
\item for all $t \in [2,T]$, $\overline{w}_t = \sum_{s \in \mathcal{S}_{t,d}} \lambda^s$, where $\mathcal{S}_{t,d} = \lbrace s \,|\,  \hat{w}^s_t = 1 \rbrace$;
\item for all $t \in [2,T]$, $\overline{u}_t = \sum_{s \in \mathcal{S}_{t,u}} \lambda^s$, where $\mathcal{S}_{t,u} = \lbrace s \,|\,  \hat{u}^s_t = 1 \rbrace$.
\end{enumerate}
\end{lemma}

\begin{proof}
We prove by induction. 

\textbf{Base case:} consider $T = 2$. Since the time resolution we consider is one dispatch interval, a minimum up/down time of less than one dispatch interval has the same effect as a minimum up/down time of one dispatch interval. Therefore, $D_2(1, 1) = D_2(1, 0) = D_2(0, 1) = D_2(0, 0)$, so that we can only consider the case $D_2(1, 1)$, which is the case considered in the original proof of \cite{rajan2005minimum}.

\textbf{Induction hypothesis:} suppose the given statement holds for $T -1$, where $T > 2$ and $T \in \mathbb{Z}$. That is, suppose $\forall L',l' \in \mathbb{Z}_{+}$ such that $L' \leq T-2$ and $l' \leq T- 2$, $\forall (\bm{x}, \bm{u}) \in D_{T-1}(L',l')$, there exists a set of integral points $\lbrace (\tilde{\bm{x}}^s, \tilde{\bm{u}}^s), s \in \mathcal{S}' = \lbrace 1, 2, \ldots, |\mathcal{S}'|\rbrace \rbrace \subset D_{T-1}(L',l') $, and $\lbrace \mu^s \in \mathbb{R}_+, s \in \mathcal{S}' \rbrace$ such that (i), (ii), and (iii) hold. 

\textbf{We need to show:} the statement holds for $T$. That is, $\forall L,l \in \mathbb{Z}_{+}$ such that $L \leq T-1$ and $l \leq T-1$, $\forall (\overline{\bm{x}}, \overline{\bm{u}}) \in D_T(L,l)$, there exists a set of integral points $\lbrace (\hat{\bm{x}}^s, \hat{\bm{u}}^s), s \in \mathcal{S} = \lbrace 1, 2, \ldots, |\mathcal{S}|\rbrace \rbrace \subset D_T(L,l) $, and $\lbrace \lambda^s \in \mathbb{R}_+, s \in \mathcal{S} \rbrace$ such that (i), (ii), and (iii) hold.

Given any $L,l \in \mathbb{Z}_{++}$ such that $L \leq T-1$ and $l \leq T-1$, and given any $(\overline{\bm{x}}, \overline{\bm{u}}) \in D_T(L,l)$, we drop the last entry of $\overline{\bm{x}}$ and denote the truncated vector by  $\bm{x} \in \mathbb{R}^{T-1}$. Similarly, we drop the last entry of $\overline{\bm{u}}$ and denote the truncated vector as  $\bm{u} \in \mathbb{R}^{T-2}$. We then have $(\bm{x}, \bm{u}) \in  D_{T-1}(L-1,l-1)$.\footnote{For the boundary cases where either $L$ or $l$ equals one, or both, note that $D_{T}(L,0) = D_{T}(L,1)$, $D_{T}(0,l) = D_{T}(1,l)$, and $D_{T}(0,0) = D_{T}(1,1)$.} Since $L-1 \leq T-2$ and $l-1 \leq T-2$, by the induction hypothesis, we can find a set of integral points $\lbrace (\tilde{\bm{x}}^s, \tilde{\bm{u}}^s), s \in \mathcal{S}' \rbrace$ and $\lbrace \mu^s \in \mathbb{R}_+, s \in \mathcal{S}' \rbrace$ that satisfy  (i), (ii), and (iii). Similarly, we let $\mathcal{S}'_{t,d} = \lbrace s \,|\,  \tilde{w}^s_t = 1 \rbrace$ and $\mathcal{S}'_{t,u} = \lbrace s \,|\,  \tilde{u}^s_t = 1 \rbrace$.

We construct $\lbrace (\hat{\bm{x}}^s, \hat{\bm{u}}^s), s \in \mathcal{S} \rbrace\subset D_T(L,l)$ from $\lbrace (\tilde{\bm{x}}^s, \tilde{\bm{u}}^s), s \in \mathcal{S}' \rbrace \subset D_{T-1}(L-1,l-1)$ by defining all but the last components of $\hat{\bm{x}}^s$ and $\hat{\bm{u}}^{s}$ to be the same as $\tilde{\bm{x}}^s$ and $\tilde{\bm{u}}^s$, respectively, and appending a $T$-th component as needed. When necessary, we may construct more than one $ (\hat{\bm{x}}^s, \hat{\bm{u}}^s)$ based on a single $(\tilde{\bm{x}}^s, \tilde{\bm{u}}^s)$. We construct $(\hat{\bm{x}}^s, \hat{\bm{u}}^s)$ in a way that allows us to find a set of $\lambda^s$ satisfying (i), (ii), and (iii).

To facilitate our proof, we partition $\mathcal{S}'$ into $\mathcal{S}'_1$ and $\mathcal{S}'_2$, so that
\begin{equation}
\mathcal{S}'_1 = \lbrace s \in \mathcal{S}' \, | \, \tilde{x}^s_{T-1} = 1 \rbrace,
\end{equation}
and 
\begin{equation}
{S}'_2 = \lbrace s \in \mathcal{S}' \, | \, \tilde{x}^s_{T-1} = 0 \rbrace.
\end{equation}
That is, $\mathcal{S}'_1$ contains integral points that involve the unit being on at time $T-1$. 

We further partition $\mathcal{S}'_1$ into $\mathcal{S}'_{11}$ and $\mathcal{S}'_{12}$, so that 
\begin{equation}
\mathcal{S}'_{11} =  \lbrace s \in \mathcal{S}'_1 \, | \, \tilde{u}_{T-L+1}^s = \tilde{u}_{T-L+2}^s = \cdots  = \tilde{u}_{T-1}^s = 0 \rbrace.
\end{equation}
and $\mathcal{S}'_{12} = \mathcal{S}'_1 \backslash \mathcal{S}'_{11}$. That is, the integral points in $\mathcal{S}'_{11}$ involve the unit not starting up during $t \in [T-L+1, T-1]$. 

Similarly, we partition $\mathcal{S}'_2$ into $\mathcal{S}'_{21}$ and $\mathcal{S}'_{22}$, so that 
\begin{equation}
\mathcal{S}'_{21} =  \lbrace s \in \mathcal{S}'_1 \, | \, \tilde{w}_{T-l+1}^s = \tilde{w}_{T-l+2}^s = \cdots  = \tilde{w}_{T-1}^s = 0  \rbrace,
\end{equation} 
and $\mathcal{S}'_{22} = \mathcal{S}'_2\backslash \mathcal{S}'_{21}$. That is, the integral points in $\mathcal{S}'_{21}$ involve the unit not shutting down during $t \in [T-l+1, T-1]$. 

We have effectively partitioned $\mathcal{S}'$ into $\mathcal{S}'_{11}$, $\mathcal{S}'_{12}$, $\mathcal{S}'_{21}$, and $\mathcal{S}'_{22}$, which results in the following properties:
\begin{itemize}
\item $ \sum_{s \in \mathcal{S}'_1} \mu^s = x_{T-1}$, and $ \sum_{s \in \mathcal{S}'_2} \mu^s = 1 - x_{T-1}$.
\item $ \sum_{s \in \mathcal{S}'_{11}} \mu^s \geq \overline{w}_T $. 

To see this, notice that $ \sum_{s \in \mathcal{S}'_{11}} \mu^s = x_{T-1} - \sum_{t = T-L+1}^{T-1} \sum_{s \in \mathcal{S}'_{t,u}} \mu^s =  x_{T-1} - \sum_{t = T-L+1}^{T-1} u_t$, where the first equality follows the definition of $\mathcal{S}^{\prime}_{11}$ and $\mathcal{S}^{\prime}_{t, u}$, and the second equality follows from (iii). Now we invoke the turn on inequality at $T$ for $D_T(L,l)$: $\sum_{t = T - L + 1}^T \overline{u}_t = \sum_{t = T-L+1}^{T-1} u_t + \overline{u}_T \leq  \overline{x}_{T}$. Applying this turn on inequality to the previous equality yields $ \sum_{s \in \mathcal{S}'_{11}} \mu^s \geq -\overline{x}_T + \overline{x}_{T-1} + \overline{u}_T$. The desired result follows from \eqref{eq:shutdown}.

\item $ \sum_{s \in \mathcal{S}'_{21}} \mu^s \geq \overline{u}_T $. 

To see this, notice that $ \sum_{s \in \mathcal{S}'_{21}} \mu^s = 1 - x_{T-1} - \sum_{t = T-l+1}^{T-1} \sum_{s \in \mathcal{S}'_{t,d}} \mu^s =  1 - x_{T-1} - \sum_{t = T-l+1}^{T-1} w_t$, where the first equality is by definition of $\mathcal{S}^{\prime}_{21}$ and $\mathcal{S}^{\prime}_{t, d}$, and the second equality follows from (ii). Now we invoke the turn off inequality at $T$ for $D_T(L,l)$: $\sum_{t = T - l + 1}^T \overline{w}_t = \sum_{t = T-l+1}^{T-1} w_t + \overline{w}_T \leq 1 - \overline{x}_{T}$. Applying this turn off inequality to the previous equality yields $ \sum_{s \in \mathcal{S}'_{21}} \mu^s \geq \overline{x}_T - \overline{x}_{T-1} + \overline{w}_T$. The desired result follows from \eqref{eq:shutdown}.
\end{itemize} 

To satisfy (ii), we would like to append a one to some of the shut-down vectors $\tilde{\bm{w}}^s$ and assign positive values to their associated $\lambda^s$, so that $\overline{w}_T = \sum_{s \in \mathcal{S}_{T,d}}\lambda^s$. Because of the minimum up time constraints, we can only append a one to those $\tilde{\bm{w}}^s$ with $s \in \mathcal{S}' _{11}$.

To satisfy (iii), we would like to append a one to some of the start-up vectors $\tilde{\bm{u}}^s$ and assign positive values to their associated $\lambda^s$, so that $\overline{u}_T = \sum_{s \in \mathcal{S}_{T,u}} \lambda^s$.  Because of the minimum down time constraints, we can only append a one to those $\tilde{\bm{u}}^s$ with $s \in \mathcal{S}' _{21}$.

We construct $\lbrace (\hat{\bm{x}}^s, \hat{\bm{u}}^s), s \in \mathcal{S} \rbrace$ according to the following rules:
\begin{itemize}
\item Construct two integral points $ (\hat{\bm{x}}^{s1}, \hat{\bm{u}}^{s1}),  (\hat{\bm{x}}^{s2}, \hat{\bm{u}}^{s2})$ from each $\lbrace  (\tilde{\bm{x}}^s, \tilde{\bm{u}}^s), s \in \mathcal{S}'_{11} \rbrace$. Each shut-down vector $\hat{\bm{w}}^{s1}$ is created by appending a one to $\tilde{\bm{w}}^{s}$, and each $\hat{\bm{w}}^{s2}$ is created by appending a zero. Each start-up vector $\hat{\bm{u}}^{s1}$ and $\hat{\bm{u}}^{s2}$ can only be created by appending a zero to $\tilde{\bm{u}}^s$. Each commitment vector $\hat{\bm{x}}^{s1}$ is created by appending a zero to $\tilde{\bm{x}}^s$, and each $\hat{\bm{x}}^{s2}$ is created by appending a one. We set $\lambda^{s1} + \lambda^{s2} = \mu^s, \forall s \in \mathcal{S}'_{11}$. Because $ \sum_{s \in \mathcal{S}'_{11}} \mu^s \geq \overline{w}_T $, we can find $\lambda^{s1},  \lambda^{s2},  s \in \mathcal{S}'_{11}$ such that $\sum_{s \in \mathcal{S}'_{11}} \lambda^{s1} = \overline{w}_T$.

\item Construct only one integral point $(\hat{\bm{x}}^{s}, \hat{\bm{u}}^{s})$ from each $\lbrace  (\tilde{\bm{x}}^s, \tilde{\bm{u}}^s), s \in \mathcal{S}'_{12} \rbrace$. Because of the minimum up time constraints, we can only append a one to $\tilde{\bm{x}}^s$, a zero to $\tilde{\bm{u}}^s$, and a zero to $\tilde{\bm{w}}^s$. We set $\lambda^{s} = \mu^s, \forall s \in \mathcal{S}'_{12}$.

\item Construct two integral points $(\hat{\bm{x}}^{s1}, \hat{\bm{u}}^{s1}),  (\hat{\bm{x}}^{s2}, \hat{\bm{u}}^{s2})$ from each $\lbrace  (\tilde{\bm{x}}^s, \tilde{\bm{u}}^s), s \in \mathcal{S}'_{21} \rbrace$. We create each start-up vector $\hat{\bm{u}}^{s1}$ by appending a one to $\tilde{\bm{u}}^s$, and each $\hat{\bm{u}}^{s2}$ by appending a zero. Each shut-down vector $\hat{\bm{w}}^{s1},$ and $\hat{\bm{w}}^{s2}$ can only be created by appending a zero to $\tilde{\bm{w}}^{s}$. Each commitment vector $\hat{\bm{x}}^{s1}$ is created by appending a one to $\tilde{\bm{x}}^s$, and $\hat{\bm{x}}^{s2}$ is created by appending a zero. We set $\lambda^{s1} + \lambda^{s2} = \mu^s, \forall s \in \mathcal{S}'_{21}$. Because $ \sum_{s \in \mathcal{S}'_{21}} \mu^s \geq \overline{u}_T $, we can find $\lambda^{s1},  \lambda^{s2},  s \in \mathcal{S}'_{21}$ such that $ \sum_{s \in \mathcal{S}'_{21}} \lambda^{s1} = \overline{u}_T$.

\item Construct only one integral point $(\hat{\bm{x}}^{s}, \hat{\bm{u}}^{s})$ from each $\lbrace  (\tilde{\bm{x}}^s, \tilde{\bm{u}}^s), s \in \mathcal{S}'_{22} \rbrace$. Because of the minimum down time constraints, we can only append a zero to $\tilde{\bm{x}}^s$, a zero to $\tilde{\bm{u}}^s$, and a zero to $\tilde{\bm{w}}^s$. We set $\lambda^{s} = \mu^s, \forall s \in \mathcal{S}'_{22}$.
\end{itemize}

To summarize, we have constructed $\lbrace(\hat{\bm{x}}^s, \hat{\bm{u}}^s), s \in \mathcal{S} \rbrace = \lbrace (\hat{\bm{x}}^{s1}, \hat{\bm{u}}^{s1}), s \in \mathcal{S}'_{11} \rbrace \cup  (\hat{\bm{x}}^{s2}, \hat{\bm{u}}^{s2}), s \in \mathcal{S}'_{11} \rbrace \cup  \lbrace (\hat{\bm{x}}^s, \hat{\bm{u}}^s), s \in \mathcal{S}'_{12} \rbrace \cup  \lbrace (\hat{\bm{x}}^{s1}, \hat{\bm{u}}^{s1}), s \in \mathcal{S}'_{21} \rbrace \cup  \lbrace (\hat{\bm{x}}^{s2}, \hat{\bm{u}}^{s2}), s \in \mathcal{S}'_{21} \rbrace \cup  \lbrace (\hat{\bm{x}}^s, \hat{\bm{u}}^s), s \in \mathcal{S}'_{22} \rbrace$ and their associated $\lambda$. 

Finally, we verify (i), (ii), (iii) using the integral points and  $\lambda^s, s \in \mathcal{S}$ that we have constructed. Since we keep the first $T-1$ components of each $a^s$ to be the same as its corresponding $p^s$, by the way we construct $\lambda^s$, it suffices to verify (i), (ii), (iii) for only $t = T$.
\renewcommand{\theenumi}{(\roman{enumi})}%
\begin{enumerate}
\item We have $ \sum_{s \in \mathcal{S}} \lambda^s = 1$. It suffices to show that $\overline{x}_t =  \sum_{s \in \mathcal{S}} \lambda^s \hat{x}^{s}_{T}$,  $\overline{u}_T = \sum_{s \in \mathcal{S}} \lambda^s \hat{u}^{s}_{T}$. 

By construction, we have $ \sum_{s \in \mathcal{S}} \lambda^s \hat{x}^{s}_{T} =  \sum_{s \in \mathcal{S}'_{11}} \lambda^{s2} + \sum_{s \in \mathcal{S}'_{12}} \lambda^{s} + \sum_{s \in \mathcal{S}'_{21}} \lambda^{s1} = x_{T - 1} - \overline{w}_T + \overline{u}_T = \overline{x}_t$. 

Also, we have $\sum_{s \in \mathcal{S}} \lambda^s \hat{u}^{s}_{T} =  \sum_{s \in \mathcal{S}'_{21}} \lambda^{s1} = \overline{u}_T$.

\item By construction, we have $ \sum_{s \in \mathcal{S}_{T,d}} \lambda^s = \sum_{s \in \mathcal{S}'_{11}} \lambda^{s1} = \overline{w}_T$.

\item By construction, we have $ \sum_{s \in \mathcal{S}_{T,u}} \lambda^s = \sum_{s \in \mathcal{S}'_{21}} \lambda^{s1} = \overline{u}_T$.
\end{enumerate}

\end{proof}

\newpage
\bibliographystyle{IEEEtran}
\bibliography{library}

\end{document}